\documentclass[11pt,oneside, reqno]{amsart}
\usepackage[top=25mm, bottom=25mm, left=25mm, right=25mm]{geometry}
\usepackage{amsfonts,amssymb,amsmath,amsthm,amsxtra}
\usepackage[T1]{fontenc}
\usepackage[utf8]{inputenc}
\usepackage{mathtools}
\usepackage{dsfont}
\usepackage{enumitem}

\numberwithin{equation}{section}
\newtheorem{theorem}[equation]{Theorem}

\newtheorem{lemma}[equation]{Lemma}
\newtheorem{proposition}[equation]{Proposition}

\theoremstyle{definition}

\begin{document}

\title[A remark on Kolmogorov's theorem]{A remark on Kolmogorov's theorem}

\author{Arash Lotfi}
\address[Arash Lotfi]{
Department of Mathematics,
Rutgers University,
Piscataway, NJ 08854-8019, USA}
\email{arashlotfi92@gmail.com}

\author{Mariusz Mirek}
\address[Mariusz Mirek]{
Department of Mathematics,
Rutgers University,
Piscataway, NJ 08854-8019, USA 
\&
Instytut Matematyczny,
Uniwersytet Wroc{\l}awski,
Plac Grunwaldzki 2/4,
50-384 Wroc{\l}aw,
Poland}
\email{mariusz.mirek@rutgers.edu}

\author{Richard K. Yu}
\address[Richard K. Yu]{
Department of Mathematics,
Rutgers University,
Piscataway, NJ 08854-8019, USA}
\email{rky8@scarletmail.rutgers.edu}

\thanks{The authors were partially supported  by the NSF CAREER grant (DMS-2236493).}

\begin{abstract}
The aim of this note is to illustrate that the set of integrable functions on the torus $\mathbb{T}$ for which the Fourier series diverges almost everywhere is of the second Baire category in $L^1(\mathbb{T})$.
\end{abstract}

\date{\today}

\maketitle

\section{Introduction}
\subsection{A very brief history} In 1923, Kolmogorov \cite{Kol1}, as
a nineteen-year-old student, constructed an integrable function on the
torus $\mathbb T:=\mathbb R/\mathbb Z$ whose Fourier series diverges
almost everywhere. Three years later, Kolmogorov \cite{Kol2} improved
on this by creating an integrable function on $\mathbb T$ whose
Fourier series diverges everywhere on $\mathbb T$.  Kolmogorov's seminal papers \cite{Kol1, Kol2} initiated new directions of research in
Fourier series and engendered an extensive literature on diverging Fourier series. An important simplification and several refinements of Kolmogorov's examples appear in Stein's article \cite[Section 4, p.143]{S}, which also covers the case of the Walsh--Paley series. For a comprehensive historical background and a thorough bibliography on the subject, we refer to Ul'yanov's survey \cite{U} as well as to the books of Zygmund \cite[Chapter VIII]{Z}
and Grafakos \cite[Section 4.2]{Grafakos1}. In particular, Grafakos'
books \cite{Grafakos1, Grafakos2} offer a modern treatment to the
subject.

Kolmogorov's results arose from his attempt to disprove Luzin's conjecture from 1913, which asserts that the partial sums of the Fourier series of a
square-integrable function on $\mathbb T$ converge almost everywhere. At that time, there was a strong belief that Luzin's conjecture was likely false. 

However, four decades after Kolmogorov's papers,  Carleson  \cite{Carleson} proved Luzin's conjecture in the affirmative for $L^2(\mathbb T)$
functions and Hunt \cite{Hunt} subsequently extended this result to all
$L^p(\mathbb T)$ functions with $p \in (1, \infty)$. We refer to
Fefferman's paper \cite{Fef} and to Lacey and Thiele's paper \cite{LT1} for
different proofs of Carleson's theorem, as well as to Grafakos
\cite[Section 6]{Grafakos2}, where the Carleson--Hunt theorem is
presented in great detail for $L^p(\mathbb T)$ functions with $p \in (1, \infty)$.

Collectively, Kolmogorov's papers \cite{Kol1, Kol2} and the Carleson--Hunt theorem
\cite{Carleson, Hunt} reveal that Kolmogorov's examples are obstructions which characterize 
$L^1(\mathbb T)$-space phenomena for Fourier series. This advances a
natural question: how large topologically (in the sense of Baire
category) is the set of all integrable functions
$f \in L^1(\mathbb{T})$ whose Fourier series diverges almost
everywhere on $\mathbb{T}$?

In this short note, we give an ad-hoc argument based on Kolmogorov's methods \cite{Kol1, Kol2}, (see Proposition \ref{prop:1} below), illustrating that the set of integrable functions whose Fourier series diverge almost everywhere on $\mathbb T$ is topologically generic in $L^1(\mathbb T)$. In other words, this property holds on a residual set, i.e., a countable intersection of dense open sets. Loosely speaking, Kolmogorov's pathological examples are topologically quite common in $L^1(\mathbb T)$. 

\section*{Acknowledgments} We thank Leonidas Daskalakis, Jacek Dziuba{\'n}ski, Dariusz Kosz and Jim Wright  for helpful comments on an earlier draft of this paper. 

\subsection{Notation and statement of the main result}
Throughout the paper the set of positive integers and the set of nonnegative
integers will be denoted, respectively, by $\mathbb Z_+ \coloneqq \{1, 2, \ldots\}$ and
$\mathbb N \coloneqq \{0,1,2,\ldots\}$. The sets $\mathbb Z$, $\mathbb R$, and   $\mathbb C$ have their standard meaning. We also denote $\mathbb R_+ \coloneqq (0, \infty)$.  The torus $\mathbb{T} := \mathbb{R}/\mathbb{Z}$ is endowed with the Lebesgue measure $d\mu_{\mathbb T}(x):=dx$ inherited from the real line $\mathbb{R}$, and $\mu_{\mathbb T}(E):=\int_Edx$ denotes the Lebesgue measure of a Lebesgue measurable set $E \subseteq \mathbb{T}$.

All vector spaces here will be defined over the complex numbers $\mathbb C$.
The space of all measurable functions $f:\mathbb T\to \mathbb C$ whose modulus is integrable
with $p$-th power is denoted by $L^p(\mathbb T)$ for $p\in[1, \infty)$,
whereas $L^{\infty}(\mathbb T)$ denotes the space of all essentially bounded
measurable functions on $\mathbb T$.

Let $e(x):=e^{2\pi i x}$ for $x\in \mathbb R$. For $f\in L^1(\mathbb T)$ the Fourier coefficients are defined by
\[
\widehat{f}(n):=\int_{\mathbb T}f(x)e(-x n)dx \quad \text{ for } \quad n\in\mathbb Z.
\]
 The partial sums of the Fourier series corresponding to $f\in L^1(\mathbb T)$ are defined by
\[
S_N(f)(x):=\sum_{n\in[-N, N]\cap \mathbb Z}\widehat{f}(n)e(x n)\quad \text{ for } \quad x\in\mathbb T, \ N\in\mathbb N.
\]
Using the Dirichlet kernel $D_N(x):=\sum_{n\in[-N, N]\cap \mathbb Z}e(x n)$, we see that
\[
S_N(f)(x)=D_N*f(x)=\int_{\mathbb T}f(x-y)D_N(y)dy.
\]

The main theorem of the paper reads as follows.

\begin{theorem}
\label{thm:1}
The set  
\begin{align}
\label{eq:1}
\mathcal D:=\big\{f\in L^1(\mathbb T): \mu_{\mathbb T}(\{x\in\mathbb T: \sup_{N\in\mathbb N}|D_N*f(x)|=\infty\})=1\big\}
\end{align}
is of the second Baire category in 
$L^1(\mathbb T)$. In particular, the set of all integrable functions whose Fourier series diverges almost everywhere on $\mathbb T$ is of the second Baire category  in $L^1(\mathbb T)$.
\end{theorem}

From Theorem \ref{thm:1} it follows immediately that there exist integrable
functions whose Fourier series diverge almost everywhere, in fact with unbounded partial sums; however, our
proof should not be regarded as an alternative method for producing
Kolmogorov's examples \cite{Kol1, Kol2}, but rather as a complement to his approach,
providing a topological characterization of this phenomenon. 

Interestingly, Marcinkiewicz \cite{M} constructed an integrable function on $\mathbb T$ whose Fourier series diverges almost everywhere, yet whose partial sums are bounded at every point of $\mathbb T$.

We will prove Theorem \ref{thm:1} by showing that $\mathcal D^c$ is of the first category in the sense of Baire category, in other words it can be written as a countable union of nowhere dense sets in $L^1(\mathbb T)$. 

A key tool will be the following result of Kolmogorov \cite{Kol1, Kol2}.
\begin{proposition}
\label{prop:1}
For each $M\in\mathbb R_+$ there exists a trigonometric polynomial
$g_M$ and a measurable set $G_M \subseteq \mathbb T$ with measure
$\mu_{\mathbb T}(G_M) > 1 - 2^{-M}$ such that
$\|g_M\|_{L^1(\mathbb T)}=1$, and satisfying
\begin{align}
\label{eq:5}
\inf_{x\in G_M}\sup_{k\in\mathbb Z_+}|D_k*g_M(x)|>2^M.
\end{align}
\end{proposition}
\begin{proof}
For a detailed proof we refer to \cite[Lemma 4.2.4., p.258]{Grafakos1}.
\end{proof}
Proposition \ref{prop:1} (see also \cite[Chapter VIII]{Z}) was a key mechanism behind the construction of Kolmogorov's \cite{Kol1, Kol2} explicit examples; it will also be essential in our categorical argument.

For $m, n\in\mathbb Z_+$, define
\begin{align}
\label{eq:4}
F(m, n):=\Big\{f\in L^1(\mathbb T): \mu_{\mathbb T}(\{x\in\mathbb T: \sup_{N\in\mathbb N}|D_N*f(x)|>n\})\le  \frac{m}{m+1}\Big\}.
\end{align}

\begin{lemma}
\label{lem:1}
The sets $F(m, n)$ defined in \eqref{eq:4} are closed in $L^1(\mathbb T)$ for every $m, n\in\mathbb Z_+$.
\end{lemma}
\begin{proof}
Fix $m, n\in\mathbb Z_+$, equivalently it suffices to show that $F(m, n)^c$ is open in $L^1(\mathbb T)$. Pick $f\in F(m, n)^c$ and we show that there exists $\eta>0$ such that for any $g\in L^1(\mathbb T)$  when $\|f-g\|_{L^1(\mathbb T)} < \eta$, then $g\in F(m, n)^c$. Define two sets
\begin{align*}
E_N(f)&:=\big\{x \in\mathbb T: \sup_{1 \leq k \leq N}|D_k * f(x)|>n\big\},\\
E_\infty(f)&:=\big\{x \in\mathbb T: \sup_{k\in\mathbb N}|D_k * f(x)|>n\big\}.
\end{align*}
Observe that $E_N$ is an increasing sequence of sets, that is
$E_{N} \subseteq E_{N+1}$. Using continuity from below of the measure,
we have $\lim_{N\to \infty}\mu_{\mathbb T}(E_N(f)) = \mu_{\mathbb T}(E_\infty(f))$. Since
$f\in F(m, n)^c$, it follows that there exists some
$N \in \mathbb N$ such that $\mu_{\mathbb T}(E_N(f)) > \frac{m}{m+1}$. Define another set
\begin{align*}
E_{N, \delta}(f):=\big\{x \in\mathbb T: \sup_{1 \leq k \leq N}|D_k * f(x)| > n+\delta\big\}.
\end{align*}
We note that  $E_{N, \delta_2}(f)\subseteq E_{N, \delta_1}(f)$ for $\delta_2>\delta_1>0$. Hence, we can choose $\delta>0$ such that
$\mu_{\mathbb T}(E_{N, \delta}(f)) > \frac{m}{m+1}+\delta$. Set $\eta:= \min_{1 \leq k \leq N} \delta^22^{-k}\|D_k\|_{L^1(\mathbb T)}^{-1}$ and observe that for any $g\in L^1(\mathbb T)$  satisfying $\|f-g\|_{L^1(\mathbb T)} < \eta$, we obtain, for any $k\in\{1, \ldots, N\}$, by Chebyshev's inequality followed by Young's convolution inequality that
\begin{align*}
\mu_{\mathbb T}\big(\big\{x \in\mathbb T: |D_k *(f-g)(x)| \geq \delta\big\}\big)\le \frac{1}{\delta}\|D_k\|_{L^1(\mathbb T)}\|f-g\|_{L^1(\mathbb T)}<\frac{\delta}{2^k}.
\end{align*}
 Taking
\begin{align*}
B:=\bigcup_{k=1}^N\{x\in \mathbb T:\left|D_k *(f-g)(x)\right| \geq \delta\},
\end{align*}
we see that $\mu_{\mathbb T}(B)\le \sum_{k=1}^N\frac{\delta}{2^k}<\delta$. Moreover,  we have
\begin{align*}
E_{N, \delta}(f)\subseteq \big\{x \in\mathbb T: \sup_{1 \leq k \leq N}|D_k * g(x)|>n\big\}\cup B.
\end{align*}
Therefore, $\mu_{\mathbb T}(E_{N}(g))\ge \mu_{\mathbb T}(E_{N, \delta}(f))-\mu_{\mathbb T}(B)> \frac{m}{m+1}+\delta-\delta=\frac{m}{m+1}$, which implies that $g\in F(m, n)^c$. This in turn yields that $F(m, n)^c$ is open and the proof of the lemma follows. 
\end{proof}

\begin{lemma}
\label{lem:2}
The sets $F(m, n)$ defined in \eqref{eq:4} have empty interiors for every $m, n\in\mathbb Z_+$.
\end{lemma}
\begin{proof}
We fix $m, n\in\mathbb Z_+$ and we will show that $F(m, n)$ has empty interior. Equivalently, we prove that for all $f \in F(m, n)$ and for all $\varepsilon>0$, there exists $h \in L^1(\mathbb T)$ such that $\|f-h\|_{L^1(\mathbb T)} < \varepsilon$ and $h \not \in F(m, n)$.  By a simple density argument, there exists a trigonometric polynomial  $g\in L^1(\mathbb T)$ such that $\|f-g\|_{L^1(\mathbb T)} < \varepsilon/2$. Then there exists a finite constant $K\in\mathbb R_+$ such that
\begin{align*}
\sup_{x\in\mathbb T}\sup _{N \in \mathbb Z_+}|D_N * g(x)| \leq K.
\end{align*} 
Choose $M\in\mathbb Z_+$ so that 
\begin{align*}
2^{-M/2} < \frac{\varepsilon}{2},
\quad \text{ and } \quad
2^{M/2} > n+K,
\quad \text{ and } \quad
1-2^{-M}>\frac{m}{m+1}.
\end{align*}
Define the perturbed function
$h:= g + 2^{-M/2}g_M$, where $g_M$ is a polynomial as in Proposition \ref{prop:1}. By our choice of $M\in\mathbb Z_+$ we see that $\|f-h\|_{L^1(\mathbb T)} < \varepsilon$, since $\|g_M\|_{L^1(\mathbb T)}=1$ and
\begin{align*}
\|f-h\|_{L^1(\mathbb T)} \leq  \|f-g\|_{L^1(\mathbb T)}  + \|g-h\|_{L^1(\mathbb T)}  =  \|f-g\|_{L^1(\mathbb T)}+ 2^{-M/2} \|g_M\|_{L^1(\mathbb T)} < \frac{\varepsilon}{2}+ 2^{-M/2}.
\end{align*}
It suffices to show that $h \not \in F(m, n)$. Indeed, if $x \in  G_M$, then by the triangle inequality, we have
\begin{align*}
\sup_{k\in\mathbb Z_+}|D_k * h(x)|&=\sup_{k\in\mathbb Z_+}|D_k * (g+2^{-M / 2} g_M)(x)| \\
&\geq 2^{-M / 2} \sup_{k\in\mathbb Z_+}|D_k *  g_M(x)|-\sup_{k\in\mathbb Z_+}|D_k * g(x)| \\
& > 2^{-M/2} 2^M - K = 2^{M/2} -K.
\end{align*}
By our choice of $M\in\mathbb Z_+$, we conclude that
\begin{align*}
\sup_{k\in\mathbb Z_+}|D_k * h(x)| > n,
\end{align*}
which also guarantees that  $G_M \subseteq \{x\in\mathbb T : \sup_{k\in\mathbb Z_+}|D_k  * h)(x)| >n\}$, and yields
\begin{align*}
\mu_{\mathbb T}\big(\big\{x \in\mathbb T: \sup_{k\in\mathbb Z_+}|D_k  * h(x)|>n\big\}\big)\ge \mu_{\mathbb T}(G_M)>1-2^{-M}>\frac{m}{m+1}.
\end{align*}
Therefore, $h \not \in F(m, n)$ and consequently  $F(m, n)$ has empty interior as desired.
\end{proof}

\begin{proof}[Proof of Theorem \ref{thm:1}]
Observe that
\begin{align*}
\mathcal D^c= \big\{f\in L^1(\mathbb T): \mu_{\mathbb T}(\{x\in\mathbb T: \sup_{N\in\mathbb N}|D_N*f(x)|=\infty\})<1\big\}.
\end{align*}
For $f\in L^1(\mathbb T)$, define
\begin{align*}
A_n:=\{x \in\mathbb T: \sup _{N \in \mathbb{N}}|D_N * f(x)|>n\} \quad \text{ for } \quad n\in\mathbb Z_+.
\end{align*}
Observe that $A_n$ is a decreasing sequence of sets, that is $A_{n+1} \subseteq A_n$. Further $\mu_{\mathbb T}(A_1) \leq 1$. So, we may invoke continuity from above of the measure to see that
\begin{align*}
\mu_{\mathbb T}(\{x\in\mathbb T: \sup_{N\in\mathbb N}|D_N*f(x)|=\infty\})=\lim_{n\to \infty}\mu_{\mathbb T}(A_n)\le 1.
\end{align*}
Hence, we obtain
\begin{align*}
\mathcal D^c\subseteq \bigcup_{n\in\mathbb Z_+} \big\{f\in L^1(\mathbb T): \mu_{\mathbb T}(\{x\in\mathbb T: \sup_{N\in\mathbb N}|D_N*f(x)|>n\})<1\big\}.
\end{align*}
Consequently, using the sets $F(m, n)$ from \eqref{eq:4}, we can further write
\begin{align}
\label{eq:3}
\mathcal D^c\subseteq \bigcup_{n\in\mathbb Z_+}\bigcup_{m\in\mathbb Z_+} F(m, n).
\end{align}
The set $\bigcup_{n\in\mathbb Z_+}\bigcup_{m\in\mathbb Z_+} F(m, n)$ is of the first category, since each
$F(m, n)$ is closed and has empty interior by Lemma \ref{lem:1} and Lemma \ref{lem:2}.
This ensures that $\mathcal D^c$ is of the first
category as a subset of a first category set  in \eqref{eq:3}. 
The proof of Theorem \ref{thm:1} now follows.
\end{proof}

\end{document}